\documentclass[12pt, reqno]{amsart}

\usepackage{amsmath,amsthm,amsfonts,amscd,amssymb}
\usepackage{times,euler,eucal,eufrak,graphicx}
\usepackage{hyperref}

\usepackage{caption} 
\usepackage{subcaption} 

\title[Haagerup's inequality and additivity violation of the MOE]
{Haagerup's inequality and additivity violation of the Minimum Output Entropy}
\author {Beno\^\i{}t Collins}
\address{Department of Mathematics, Kyoto University} \email{collins@math.kyoto-u.ac.jp}

\theoremstyle{plain}
\newtheorem{lemma}{Lemma}[section]
\newtheorem{theorem}[lemma]{Theorem}
\newtheorem{proposition}[lemma]{Proposition}

\theoremstyle{definition}

\theoremstyle{remark}

 \DeclareMathOperator{\Tr}{Tr}

\newcommand{\C}{{\mathbb{C}}}

\begin{document}

\begin{abstract}
We give a simple and conceptual proof of the fact that random unitary channels yield violation of 
the Minimum Output Entropy additivity. The proof relies on strong convergence of random 
unitary matrices
and Haagerup's inequality. 
\end{abstract}

\maketitle

\section{Introduction}

During the last decade, a crucial problem in quantum information was to determine whether one can find two quantum channels
$\Phi_i: B(H_{j_i})\to B(H_{k_i}), i=\{1,2\}$, such that 
$$H_{min}(\Phi_1\otimes\Phi_2)<H_{min}(\Phi_1)+H_{min}(\Phi_2).$$
We refer to section \ref{sec:notation} for definitions.

This problem was known as the \emph{MOE additivity problem}, and it was crucial because 
it is equivalent to the additivity of the Holevo capacity, as well as other quantities. 
We refer to the survey \cite{CN2016}
and bibliography therein for references.
Additivity was proven to be false by  Hastings \cite{Hastings2009}.
Later, generalizations and improvements were made 
in \cite{FKM2010,BrandaoHorodecki2010,FK2010,AubrunSzarekWerner2011,Fukuda2014,BCN2013}.

Each proof of the above result relies on specific random counterexamples. The initial counterexamples
required the development of highly specific counterexamples (tubular neighbourhoods \cite{Hastings2009,FKM2010}), whereas
later on, it was realized that there are relation with conceptual tools already available
 (Dvoretzky's theorem \cite{AubrunSzarekWerner2011}, large deviation theory \cite{BrandaoHorodecki2010}).

The proofs that achieve the smallest output dimension and the largest violation rely on free probability
theory, and required the development of specific tools (\cite{BCN2013}). They were specific to 
quantum channels obtained with random Stinespring isometries, and did not work
with the initial counterexample of Hastings -- random unitary channels.

The purpose of this note is to fill this gap, and explain from a simple conceptual point of view why
random unitary channels yield violation, and in the meantime, obtain an (almost) elementary 
proof with operator algebras. 
The main tool we use is Haagerup's inequality, which we recall in section \ref{sec:haagerup}.

\emph{Organization of the paper}:
This paper is organized as follows. After this introductory part, section \ref{sec:preliminary} sets up some notations
and gathers prerequisite results. Finally section \ref{sec:main} provides the main estimate and the main result. 

\emph{Acknowledgements}:
The author would like to thank Ion Nechita for  discussions
and Motohisa Fukuda for constructive
comments on a preliminary version of this draft. 
The author was supported by NSERC, JSPS Kakenhi,  and ANR-14-CE25-0003.

\section{Preliminary material}\label{sec:preliminary}

\subsection{Notations of Quantum Information Theory}\label{sec:notation}

We denote by $H$ an Hilbert space, which we assume to be finite dimensional. 
$B(H)$ is the set of bounded linear operators on $H$, and 
$D(H)\subset B(H)$ is the collection of trace 1, positive operators -- known as
\emph{density matrices}.

For $X\in D(H)$, its von Neumann entropy is defined by functional
calculus by $H(X)=-\Tr X\log X$, where $0\log 0$ is assumed by continuity to be zero.

A quantum channel $\Phi : B(H_1)\to B(H_2)$ is a completely positive trace preserving linear map. 
The Minimum Output Entropy of $\Phi$ is
$$H_{min}(\Phi)=\min_{X\in D(H_1)} H(\Phi (X)).$$
We refer to the survey \cite{CN2016} for further properties.

\subsection{Strong convergence}
A \emph{$*$-non-commutative tracial probability space} is a pair $(\mathcal{A},\tau)$ where
$\mathcal{A}$ is a unital $*$-algebra, and $\tau$ is a positive trace satisfying $\tau (1)=1$.

We say that a sequence of $k$-tuples $(a_i^{(n)})_{i=1}^k$ in a sequence of 
$*$-non-commutative probability spaces
$(\mathcal{A}_n,\tau_n)$,  converges \emph{in distribution} to 
the distribution of  $(a_1,\ldots ,a_k)\in (\mathcal{A},\tau)$ iff 
for any free $*$-word $w$, in $k$ variables, 
$\tau_n (w(a_i^{(n)}))\to \tau (w(a_i))$.

Likewise, a sequence is said to converge \emph{strongly in distribution} iff it converges in distribution, and 
in addition, for any non-commutative polynomial $P$, its operator norm converges
$$\|P(a_1^{(n)},\ldots ,a_k^{(n)})\|\to \|P(a_1,\ldots ,a_k)\|.$$
In this definition, we assume that the operator norm is given by the distribution, i.e.
$$\|P(a_1^{(n)},\ldots ,a_k^{(n)})\|=\lim_{q\to\infty}\|P(a_1^{(n)},\ldots ,a_k^{(n)})\|_q,$$ 
and 
\begin{equation}\label{norm-limit}
\|P(a_1,\ldots ,a_k)\|=\lim_{q\to\infty}\|P(a_1,\ldots ,a_k)\|_q,
\end{equation}
where the $q$-norm $\|X\|_q$ is by definition $(\tau ((XX^*)^{q/2}))^{2/q}$ (resp. 
$(\tau_n ((XX^*)^{q/2}))^{2/q}$)

Let  $(a_i^{(n)})_{i=1}^k$ be a sequence of $n\times n$ matrices, viewed as elements of the non-commutative
probability space $(M_n,n^{-1}Tr)$ and assume that it converges strongly in distribution towards
a $k$-tuple of random variables
$(a_1,\ldots ,a_k)\in (\mathcal{A},\phi)$,
then 
\begin{theorem}\label{thm:cm}
Let $U_n$ be an $n\times n$ Haar distributed unitary random matrix (independent from
$(a_i^{(n)})_{i=1}^k$). Then the family
$$(a_1^{(n)},\ldots ,a_k^{(n)},U_n)$$
almost surely converges strongly too, towards the $k+2$-tuple of random variables
$(a_1,\ldots ,a_k,u)$, where $u$ are is a Haar unitary element, free from $(a_1,\ldots ,a_k)$.
\end{theorem}
Historically, the convergence of distribution is due to Voiculescu, \cite{voiculescu98}. 
A simpler proof was given by \cite{collins-imrn}.
The strong convergence was established in \cite{cm}, which itself heavily relies on earlier works by \cite{HT} and \cite{male}.

\subsection{Haagerup's inequality}\label{sec:haagerup}

The following inequality is due to Haagerup \cite{Haagerup1979}
and it plays a crudial role in operator algebra and free probability theory.

\begin{theorem}\label{thm:haagerup-inequality}[\cite{Haagerup1979}, Lemma 1.4]
Let $F_k$ be the free group on $k$ generators, and let $f$ be a function in $l^2 (F_k)$ supported on 
the finite subspace generated by words in $F_k$ of length $n$. 
We can see it as an element of $B(l^2 (F_k))$ and consider its operator norm $||f||$.
The following holds true:
$$||f||\leq (n+1) ||f||_2.$$
\end{theorem}

One feature of this note is to show that this Haagerup's inequality also plays a crucial role in Quantum Information Theory. 

\section{Main result}\label{sec:main}

We now describe the model that we use to produce a counterexample. 

\subsection{The model}

Let $U_1^{(n)},\ldots , U_k^{(n)}$ be iid Haar distributed unitaries in $\mathcal U_n$.
We consider the random isometry 
$$W_{k,n}: \C^n\to \C^n\oplus \ldots \oplus \C^n\cong \C^n\otimes \C^k$$
given by
$$x\mapsto k^{-1/2} (U_1^{(n)}x\oplus \ldots \oplus U_k^{(n)} x).$$
Let $\Phi_{k,n}$ be the random channel given by 
$$X\mapsto (\Tr_n \otimes Id_k)W_{k,n}XW_{k,n}^*.$$
In other words,
$$\Phi_{k,n}(X)=(\Tr (U_i^{(n)}XU_j^{(n)*}))_{ij}.$$
It is interesting to note that this channel is the complement of the
random unitary channel
$$X\to k^{-1}\sum_{i=1}^kU_i^{(n)}XU_i^{(n)*}.$$

\subsection{Main Estimate}

\begin{theorem}\label{main-estimate}
With probability one as $n\to \infty$ (with $k$ fixed)
\[
\lim_{n \to \infty} \max_{X \in D(\mathbb{C}^n)} ||\Phi_{k,n} (X)-\tilde I ||_2 \leq \frac{3}{k}.
\]
\end{theorem}

\begin{proof}
For $A=(a_{ij}) \in M_k(\mathbb C)$ and $X \in D(\mathbb{C}^n)$, 
\begin{align}\label{inequ-key}
\Tr \left[ \Phi_n(X) A\right]& = \Tr \left[ \Tr_{\mathbb C^n} [W_{k,n} X W_{k,n}^* ]A\right] 
=  \Tr \left[  W_{k,n} X W_{k,n}^*(A\otimes I_n )\right] \\
& \leq \left\| W_{k,n}^* (A\otimes I_n )W_{k,n}\right\| = 
k^{-1}||\sum_{i,j} a_{ij}U_i^{(n)}U_j^{(n)*}|| 
\end{align}
Let us define 
$$|||A|||= k^{-1}||\sum_{i,j} a_{ij}u_iu_j^*|| .$$
This quantity can be checked to be a norm on $M_k(\mathbb C)$; as a matter of fact,
it plays a role similar to the  $t$-norm introduced in in \cite{BCN2013} in the context of random quantum channels.

By Theorem \ref{thm:cm}, with probability one on any  space of sequences of random matrices
having the appropriate marginals, we have:
$$\lim_n ||\sum_{i,j} a_{ij}U_i^{(n)}U_j^{(n)*}||\to |||(a_{ij})|||.$$
By the Haagerup inequality (Theorem \ref{thm:haagerup-inequality})
the right hand side is bounded as follows: 
$|||A|||\leq k^{-1}|\Tr A | + \frac{3}{k}\sqrt{\sum_{i\neq j}|a_{ij}^2|}.$ 
In particular, in the case of traceless matrices, 
$$|||A|||\leq \frac{3}{k}||A||_2.$$

By a standard compactness and continuity argument, 
the inequality \eqref{inequ-key}
 holds uniformly on any bounded choice of $A$.
 We refer for example to Proposition 7.3 of 
 \cite{CFZ2015} 
for details. 

In turn, taking $A= \Phi_n (X)-\tilde I$,
we get, with probability one, for any $\varepsilon >0$, $n$ large enough,
\begin{align*}
||\Phi_{n,k}(X)-\tilde I||_2^2=\Tr ((\Phi_{n,k}(X)-\tilde I)^2)\leq \Tr ((\Phi_{n,k}(X)-\tilde I)\Phi_{n,k}(X)) \\
\leq|||\Phi_{n,k}(X)-\tilde I |||\leq \frac{3}{k} ||\Phi_{n,k}(X)-\tilde I ||_2 (1+\varepsilon)
\end{align*}
We refer to Equation (18) of \cite{CFZ2015} 
for further detail. Dividing both sides of the above inequality by
$||\Phi_{n,k}(X)-\tilde I ||_2$ implies, that 
with probability one,
$$\limsup_n ||\Phi_{n,k}(X)-\tilde I||_2\leq \frac{3}{k}, $$
which is what we needed.

\end{proof}

\subsection{Application to violation}

We recall the definition of the conjugate
of the quantum channel $\Phi_{n,k}$: 
$$\overline \Phi_{n,k} : X\to k^{-1}\sum_{i=1}^k \overline{U_i^{(n)}}X U_i^{(n)t}.$$
We first recall the following
\begin{proposition}\label{thm:product-bound}
The following holds true
\[
H^{\min} (\Phi_{n,k}\otimes \overline \Phi_{n,k} )  \leq 2\log k -\frac{\log k}{k}
\]
\end{proposition} 
For the proof, we refer to the original paper by Hastings, or \cite{FKM2010}, section 5.1.
Second, we quote the following bound:
\begin{align}\label{eq:concave}
\log k - H(X)  \leq k \cdot \Tr (X-\tilde I)^2 
\end{align}
for $X \in D (\mathbb C^n)$.
Here, we refer to  \cite{Hastings2009} (see also for example Lemma 2.2 of \cite{Fukuda2014}). 
In turn, Equation \eqref{eq:concave}, together with Theorem \ref{main-estimate} yields the following 
\begin{proposition}\label{thm:single-bound}
\[
H^{\min} (\Phi) \geq \log k-9/k
\]
\end{proposition} 
Therefore, putting things together, we get
\begin{theorem}
With probability one, for $n$ large enough, there is a violation for the unitary quantum channel, for 
$k$ large enough. 
\end{theorem}

\begin{proof}
According to Proposition \ref{thm:product-bound} and Equation \eqref{eq:concave},
we need to ensure that 
$$2\log k -\frac{\log k}{k}< 2(\log k-\frac{9}{k}).$$
This is equivalent to
$$\log k\geq 18,$$
and the inequality will hold as soon as $k\geq e^{18}$.
\end{proof}

As a conclusion, let us remark that 
the bound is not as good as the one obtained in \cite{BCN2013}, or many other paper of the bibliography.
However, the only proofs of operator algebraic flavour so far did not allow for random unitary channels.
The primary interest of the proof of this paper is that it shows that Haagerup's inequality can be expected to
play a more important role in the study of random quantum channels. 

Let us also add that there has been a line of research (\cite{HLSW04,Au09}) where it was proved that 
randomizing channels in dimension $n$  with $k$ Haar unitary 
independent operators sends all states to a density matrix
whose operator norm is less than $C/k$ where $C$ is a universal constant. 
Although the previous papers were considering the case where both $k,n\to\infty$, we focus on the
case where only $n\to \infty$ for a fixed value of $k$. In this case, 
it is actually possible to compute precisely the optimal constant, namely, $C=4(k-1)/k$.
More importantly, we can give precise information about the $L^2$-distance to the maximally mixed state.
Specifically, the previous results (and the triangle inequality) imply that with overwhelming 
probability, they remain within $L^2$ distance $9/k$ of the maximally mixed state. 
This rules out, for example, the possibility that many eigenvalues of a mixed state be large (say, more than $3/k$).
Such results did not follow from \cite{HLSW04,Au09}.


\begin{thebibliography}{FKM10}

\bibitem[Au09]{Au09}
Guillaume Aubrun.
\newblock On almost randomizing channels with a short Kraus decomposition. 
\newblock {\em Comm. Math. Phys.}, 288 no. 3, 1103--1116, 2011

\bibitem[ASW11]{AubrunSzarekWerner2011}
Guillaume Aubrun, Stanis{\l}aw Szarek, and Elisabeth Werner.
\newblock Hastings's additivity counterexample via {D}voretzky's theorem.
\newblock {\em Comm. Math. Phys.}, 305(1):85--97, 2011.

\bibitem[BCN13]{BCN2013}
Serban~T. Belinschi, Beno\^it Collins, and Ion Nechita.
\newblock Almost one bit violation for the additivity of the minimum output
  entropy.
  \newblock {\em Comm. Math. Phys.}, to Appear {\em arXiv:1305.1567 [math-ph]}, 2013.

\bibitem[BH10]{BrandaoHorodecki2010}
Fernando G. S.~L. Brand{\~a}o and Micha{\l} Horodecki.
\newblock On {H}astings' counterexamples to the minimum output entropy
  additivity conjecture.
\newblock {\em Open Syst. Inf. Dyn.}, 17(1):31--52, 2010.

\bibitem[C03]{collins-imrn}
Collins, B.
\newblock Moments and Cumulants of Polynomial random variables on unitary groups, 
the Itzykson-Zuber integral and free probability 
\newblock {\em Int. Math. Res. Not.}, (17):953-982, 2003. 

\bibitem[CFZ16]{CFZ2015}
Beno\^it Collins, Motohisa Fukuda, Ping Zhong
\newblock Estimates for compression norms and additivity
violation in quantum information
\newblock {\em International Journal of Mathematics}
Vol. 26, No. 1 (2015) 1550002

\bibitem[CoMa14]{cm}
Collins, B. and Male, C.
\newblock The strong asymptotic freeness of Haar and deterministic matrices
\newblock {\em Ann Sci ENS}
s\'erie 4 47, fascicule 1 (2014) 

\bibitem[CN16]{CN2016}
Beno\^it Collins, and Ion Nechita.
\newblock Random matrix techniques in quantum information theory
\newblock {\em J. Math. Phys.}, 57, 015215 (2016); doi: 10.1063/1.4936880

\bibitem[FK10]{FK2010}
Motohisa Fukuda and Christopher King.
\newblock Entanglement of random subspaces via the {H}astings bound.
\newblock {\em J. Math. Phys.}, 51(4):042201, 19, 2010.


\bibitem[FKM10]{FKM2010}
Motohisa Fukuda, Christopher King, and David~K. Moser.
\newblock Comments on {H}astings' additivity counterexamples.
\newblock {\em Comm. Math. Phys.}, 296(1):111--143, 2010.

\bibitem[Fuk14]{Fukuda2014}
Motohisa Fukuda.
\newblock Revisiting additivity violation of quantum channels.
\newblock {\em Comm. Math. Phys.},  332, 2, 713--728, 2014.

\bibitem[Haa79]{Haagerup1979}
Uffe Haagerup.
\newblock An example of a non-nuclear $C^*$-algebra which has the metric approximation property
\newblock {\em Invent. Math.}, 50: 279--293, 1979.

\bibitem[HT05]{HT}
Haagerup, U. and Thorbj\o{}rnsen, S. 
\newblock A new application of random matrices: ${\rm Ext}(C\sp *\sb {\rm red}(F\sb 2))$ is not a group.
\newblock {\em Ann. of Math.} (2)  162  (2005),  no. 2, 711--775.

\bibitem[Has09]{Hastings2009}
Matthew~B. Hastings.
\newblock Superadditivity of communication capacity using entangled inputs.
\newblock {\em Nature Physics}, 5:255, 2009.

\bibitem[HLSW04]{HLSW04}
Patrick Hayden, Debbie Leung, Peter W. Shor, Andreas Winter
\newblock Randomizing quantum states: Constructions and applications
\newblock {\em Comm. Math. Phys.} 250(2):371--391, 2004.

\bibitem[Mal11]{male}
C. Male
\newblock The norm of polynomials in large random and deterministic matrices.
\newblock {\em Probability Theory and Related Fields}, pages 1--56, June 2011.

\bibitem[Voi98]{voiculescu98}
D. Voiculescu.
\newblock A strengthened asymptotic freeness result for random matrices with applications to 
freeentropy.
\newblock {\em Internat. Math. Res. Notices}, (1):41--63, 1998.

\end{thebibliography}
\end{document}